\documentclass[12pt,a4paper,twoside, reqno, tbtags]{article}
\usepackage{amsmath,amsthm,amssymb,amsfonts,graphicx,geometry,epsfig,cases,subcaption,amssymb,amscd,enumerate,verbatim}
\usepackage[colorlinks=true, linkcolor=blue, urlcolor=blue, citecolor=blue]{hyperref}

\usepackage{tikz}

\numberwithin{equation}{section}
\newtheorem{theorem}{Theorem}[section]
\newtheorem{corollary}{Corollary}[section]
\newtheorem{lemma}{Lemma}[section]
\newtheorem{definition}{Definition}[section]
\newtheorem{proposition}{Proposition}[section]
\newtheorem{example}{Example}[section]
\newtheorem{remark}{Remark}[section]
\usepackage{sectsty}
\let\OLDthebibliography\thebibliography
\renewcommand\thebibliography[1]{
	\OLDthebibliography{#1}
	\setlength{\parskip}{0pt}
	\setlength{\itemsep}{0pt plus 0.3ex}
}

\usepackage{titlesec}
\titlespacing*{\section}
{0pt}{1ex}{0ex}
\titlespacing*{\subsection}
{0pt}{0ex}{0ex}
\titlespacing*{\subsubsection}
{0pt}{0ex}{0ex}
\usepackage{enumitem}
\setlist[1]{itemsep=-5pt}
\begin{document}
	\label{pgStart}
\begin{center}
    \textbf{Multiple Cylinder Relations for Finite Spaces and Nerve Theorem for Strong-good cover}

    By
 
    \textbf{Ponaki Das$^1$}, \textbf{Sainkupar Marwein Mawiong$^{2*}$}
    
    $^1$Department of Mathematics, North-Eastern Hill University, NEHU Campus, Shillong, India-793022
    
    $^2$Department of Basic Sciences and Social Sciences, North-Eastern Hill
    University,
    
    NEHU Campus, Shillong, India-793022
    
    Email: ponaki.das20@gmail.com $^*$Corresponding author: skupar@gmail.com
\end{center}

\begin{abstract}
\noindent In this paper, we introduce the \emph{multiple cylinder of relations}, a generalization of the relation cylinder that extends the concept of the multiple non-Hausdorff mapping cylinder to sequences of finite $T_0$-spaces connected by a series of relations. This construction captures complex homotopical structures across chains of finite spaces and, when the relations are induced by maps, serves as an intermediate space that collapses onto two distinct finite spaces.

\noindent We also define the notion of a \emph{strong-good cover} for both finite simplicial complexes and finite spaces. Unlike classical good covers, which require all finite non-empty intersections to be contractible, strong-good covers impose a stricter condition: that these intersections be \emph{collapsible}. This leads to a strengthened version of the classical \emph{Nerve Theorem}, proving that finite simplicial complexes or finite spaces with strong-good covers are \emph{(simple) homotopy equivalent} to their nerves.

\noindent Furthermore, we demonstrate that strong-good covers yield significant implications in persistent homology. Specifically, we show that the associated persistence modules vanish in all positive degrees, illustrating that collapsibility enforces strong combinatorial constraints with topological consequences.
\end{abstract}

\begin{flushleft}
\end{flushleft}

\section{Introduction}

The homotopical study of finite \( T_0 \)-spaces has gained significant interest in algebraic topology, combinatorics, and discrete geometry due to the versatility of finite spaces as combinatorial models that reflect the structure of classical topological spaces. Finite \( T_0 \)-spaces enable discrete analogs of continuous constructions, offering computationally feasible frameworks for examining homotopy types. One such foundational construction is the \emph{non-Hausdorff mapping cylinder}, introduced by Barmak~\cite{Barmak}, which adapts the classical mapping cylinder to the setting of finite $T_0$-spaces. This tool has been used to establish simplified proofs of central results, such as Quillen's Theorem A for posets, and to refine the classical Nerve Theorem for finite topological spaces.

\noindent The \textit{Nerve Theorem} is a fundamental idea in algebraic topology which has been widely applied across various fields due to its ability to relate the homotopy type of a space \( X \) to the nerve of a cover \( U = \{ U_i \} \) of $X$. Originally formulated by Borsuk \cite{Borsuk} and later developed by Leray \cite{Leray}, Weil \cite{Weil}, and McCord \cite{McCord}, the Nerve Theorem states that if the cover \( U \) is good, meaning every intersection of elements in \( U \) is contractible or empty, then the nerve of \( U \) has the same homotopy type as \( X \). More recently, Björner \cite{Bjorner} and Barmak \cite{Barmak} expanded on these results, finding applications in combinatorics, computational topology, and applied topology. However, the condition of being a good cover is somewhat limiting and not always practical.

\noindent In the homological setting, the classical Nerve Theorem was refined by Meshulam~\cite{Meshulam}, who established a \textit{Homological Nerve Theorem} that extends the classical result to more general algebraic contexts. Specifically, let \(\mathcal{U} = \{U_i\}_{i \in I}\) be a finite cover of a simplicial complex \(X\), and let \(N = N(\mathcal{U})\) denote the nerve of the cover. If for some integer \(k \geq 0\), the reduced homology \(\widetilde{H}_j(U_\sigma)\) vanishes for all \(\sigma \in N^{(k)}\) and all \(0 \leq j \leq k - \dim \sigma\), then the reduced homology of \(X\) is isomorphic to that of the nerve in degrees up to \(k\); that is, \(\widetilde{H}_j(X) \cong \widetilde{H}_j(N)\) for all \(0 \leq j \leq k\). Moreover, if \(H_{k+1}(N) \neq 0\), then \(H_{k+1}(X) \neq 0\). This homological refinement, which relies solely on the acyclicity of finite intersections, significantly broadens the applicability of nerve theorems within algebraic topology.

\noindent This paper builds on the Nerve Theorem by introducing a new concept called the ``multiple cylinder of relations". This idea generalizes the non-Hausdorff mapping cylinder, allowing sequences of finite \( T_0 \)-spaces to be connected not just by direct maps but through broader relational structures. This approach makes it possible to study homotopy equivalences between finite spaces linked by general relations, expanding the range of homotopy methods in combinatorial topology. The multiple cylinder of relations captures homotopical structures across chains of related finite spaces, creating a unified way to analyze homotopy equivalences and mapping properties. When relations are induced by maps, the multiple cylinder of relations simplifies to an intermediary space that collapses to two distint finite spaces (Remark \ref{2}).

\noindent Moreover, this paper introduces the concept of \textit{strong-good covers} for simplicial complexes and finite spaces (Definition \ref{11} for simplicial complexes and Definition \ref{14} for finite spaces). While classical good covers require that every non-empty intersections of subfamilies are contractible, strong-good covers impose a stricter condition: that every non-empty intersections are collapsible. This distinction allows for a strengthened form of the Nerve Theorem. Specifically, we establish a \textit{Nerve Theorem for strong-good covers} applicable to both finite simplicial complexes (Remark \ref{13}) and finite spaces (Remark \ref{17}), proving that simplicial complexes and their nerves (or finite spaces and their non-Hausdorff nerves) share the same (simple) homotopy type under these stronger conditions.

\noindent In addition to this homotopical refinement, we demonstrate a new application of strong-good covers in the context of persistent homology. We show that when a simplicial complex is covered by a strong-good cover, the resulting persistence modules vanish in all positive degrees. This vanishing property does not hold for classical good covers, thereby distinguishing the strong-good condition as not only a homotopically stronger assumption but also one with meaningful algebraic consequences in the study of filtrations and topological data analysis.

\noindent The main results of this paper can be summarized as follows:
\begin{enumerate}
    \item \textbf{Definition, Properties, and Collapse Conditions of the Multiple Cylinder of Relations}: We formally define the multiple cylinder of relations (Definition \ref{1}) and establish its homotopy properties, identifying conditions under which sequences of relations within the cylinder result in homotopical collapses. These collapses lead to simple homotopy equivalences between the order complexes of the finite spaces involved (Propositions \ref{3} and \ref{4}). Additionally, we determine conditions where certain homotopical trivialities in intersections lead to further collapses within the multiple cylinder of relations, simplifying order complexes and revealing insights into the homotopical structure in complex arrangements of finite spaces connected by relations (Propositions \ref{3}, \ref{4}, \ref{6}, \ref{7}, Remarks \ref{5},\ref{8}, \ref{9} and, \ref{10}).

    \item \textbf{Nerve Theorem for Strong-Good Covers of Finite Spaces and Complexes}: By introducing the notion of strong-good covers with collapsible intersections, we establish a strengthened Nerve Theorem for finite simplicial complexes (Remark~\ref{13}) and finite spaces (Remark~\ref{17}). This result shows that finite simplicial complexes (or finite spaces) and their corresponding nerves (or non-Hausdorff nerves) are (simple) homotopy equivalent, extending classical results in discrete topological settings and providing novel tools for analyzing (simple) homotopy types.

    \item \textbf{Vanishing of Persistent Homology over Strong-Good Covers}: We prove that if  a  finite simplicial complex admits a strong-good cover, then all positive-degree persistent homology groups associated with subunions of the cover vanish (Theorem~\ref{persistent-truncation}). This result, which does not hold under the usual good cover assumption, highlights a structural advantage of collapsibility in controlling algebraic invariants of filtrations.
\end{enumerate}

These results contribute new perspectives in the homotopical study of finite \( T_0 \)-spaces and finite simplicial complexes, providing generalized constructs that advance beyond classical settings. By introducing the multiple cylinder of relations and the strong-good cover framework, we enhance the toolkit of algebraic topology, offering refined methods for investigating homotopical structures in finite and combinatorial spaces. This paper thus lays foundational work for future research in combinatorial topology, discrete geometry, and related fields, where the homotopical and homological properties of finite spaces are central.

\noindent In this work, we will be using the terms ``beat points," ``weak points," ``$\gamma$- points," ``contractible," ``collapsible," and ``homotopically trivial." For detailed explanations and precise definitions of these concepts, please refer to \cite{Barmak}.

\section{Preliminaries}

    In this section, we recall the key concepts and definitions that will be used throughout our work.
\begin{definition}\cite{Barmak}\label{a}
    The \textbf{mapping cylinder} of a map $f: X \rightarrow Y$ between topological spaces is the space $Z_f$ obtained from $(X \times I) \cup Y$ by identifying each point $(x, 1) \in X \times I$ with $f(x) \in Y$. 
\end{definition}
    \noindent Both $X$ and $Y$ are subspaces of $Z_f$. We denote by $j: Y \rightarrow Z_f$ and $i: X \rightarrow Z_f$ the canonical inclusions, where $i$ is defined by $i(x) = (x, 0)$. The space $Y$ is a strong deformation retract of $Z_f$. Moreover, there exists a retraction $r : Z_f \rightarrow Y$ with $jr \simeq 1_{Z_f}$ rel $Z_f$ which satisfies that $ri = f$.
    
    \noindent Barmak introduced a finite analog of the classical mapping cylinder known as the non-Hausdorff mapping cylinder.
\begin{definition}\cite{Barmak}\label{b}
    Let $f: X \rightarrow Y$ be a map between finite $T_0$-spaces. We define the \textbf{non-Hausdorff mapping cylinder} $B(f)$ as the following finite $T_0$-space. The underlying set is the disjoint union $X \cup Y$. We keep the given ordering within $X$ and $Y$ and for $x \in X$, $y \in Y$ we set $x \leq y$ in $B(f)$ if $f(x) \leq y$ in $Y$.
\end{definition}
    \noindent We will denote by $i: X \rightarrow B(f)$ and $j: Y \rightarrow B(f)$ the canonical inclusions of $X$ and $Y$ into the non-Hausdorff mapping cylinder.
\begin{lemma}\cite{Barmak}\label{c}
    Let $f: X \rightarrow Y$ be a map between finite $T_0$-spaces. Then $Y$ is a strong deformation retract of $B(f)$.
\end{lemma}
    \noindent If $X$ and $Y$ are two finite spaces homotopy equivalent to each other, there exists a third space $Z$ containing both $X$ and $Y$ as strong deformation retracts. This space can be taken as the mapping cylinder of any homotopy equivalence $X \rightarrow Y$. If $f: X \rightarrow Y$ is a homotopy equivalence between finite $T_0$-spaces, $Y$ is a strong deformation retract of $B(f)$, but $X$ in general is just a (weak) deformation retract. we have seen that if two finite $T_0$-spaces are homotopy equivalent, there exists a third finite $T_0$-space containing both as strong deformation retracts (\cite{Barmak},Proposition 4.6.6.).
    
    \noindent To define multiple cylinder of relations, let's first recall the definition of relation cylinder: 
\begin{definition}[\textnormal{\cite{Ximena}}]\label{d}
    Given $R\subseteq X\times Y$, we define \textbf{the cylinder of the relation} (or \textbf{the relation cylinder}) $B(R)$ as the following finite poset. The underlying set is the disjoint union $X \sqcup Y$. We keep the given order in both
    $X$ and $Y$, and for every $x\in X$ and $y \in Y$ we set $x\leq y$ in $B(R)$ if there are points $x^\prime \in X$ and $y^\prime \in Y$
    such that $x \leq x^\prime$ in $X$, $y^\prime \leq y$ in $Y$ and $x^\prime Ry^\prime$ (i.e. we take the order relation generated by $x \leq y$ if $xRy$).
\end{definition}
    \noindent For $A\subseteq X$ and $B \subseteq Y$, we set $R(A) = \{y \in Y : xRy$ for some $x \in A\}$; $R^{-1}(B) = \{x \in X : xRy$ for some $y \in B\}$.
    
    \noindent If $A$ is a subspace of a finite space, $\overline{A} = \{x \in X : x \geq a$ for some $a \in A\} = \bigcup_{a\in A} F_a$ is the closure of $A$, and $\underline{A} = \{ x \in X : x \leq a \text{ for some } a \in A \} = \bigcup_{a \in A} U_a$ is the open hull of $A$.
    
    \noindent If $X$ and $Y$  are finite $T_0$-spaces such that $X$ and $Y$ are simple homotopy equivalent, then there exists a finite $T_0$-space $Z$ which collapses to both of them. One such space is obtained by considering the multiple non-Hausdorff mapping cylinder \cite{Barmak} which is a generalization of the non-Hausdorff mapping cylinder defined in \cite{Barmak}.
    
    \noindent Here, we recall the following three results which we will modify for composite relations as well as for multiple sequence of relations.
\begin{proposition}\cite{Ximena}\label{e}
    Let $R\subseteq X\times Y$ be a relation between two finite spaces. If $\underline{{(R)}^{-1}(U_y)}$ is homotopically trivial for every $y\in Y$ then $B(R)\diagup\hspace{-0.17 cm}\searrow X$. In particular, the order complexes $K(B(R))$ and $K(X)$ are (simple) homotopy equivalent. Moreover, if $\underline{{(R)}^{-1}(U_y)}$ is collapsible, then $B(R)\searrow X$ and hence $K(B(R))\searrow K(X)$.
\end{proposition}
\begin{proposition}\cite{Ximena}\label{f}
    Let $R\subseteq X\times Y$ be a relation between two finite spaces. If $\overline{{(R)}(F_x)}$ is homotopically trivial for every $x\in X$ then $B(R)\diagup\hspace{-0.17 cm}\searrow Y$. In particular, the order complexes $K(B(R))$ and $K(Y)$ are (simple) homotopy equivalent. Moreover, if $\overline{{(R)}(F_x)}$ is collapsible then  $B(R)\searrow Y$ and hence $K(B(R))\searrow K(Y)$.
\end{proposition}
\begin{corollary}\cite{Ximena}\label{g}
    Let $R\subseteq X\times Y$ be a relation between two finite spaces. If $\underline{{(R)}^{-1}(U_y)}$ and $\overline{{(R)}(F_x)}$ are homotopically trivial for every \( x \in X \), \( y \in Y \), then \( X \diagup\hspace{-0.17 cm}\searrow Y \). Moreover, if $\underline{{(R)}^{-1}(U_y)}$ and $\overline{{(R)}(F_x)}$ are collapsible, then \( K(X) \diagup\hspace{-0.17 cm}\searrow^n K(Y) \), with \( n = h(B(R)) \) being the height of the multiple cylinder of relations.
\end{corollary}
    \noindent Barmak extended the concept of the non-Hausdorff mapping cylinder for a pair of finite spaces to a sequence of functions between finite $T_0$ spaces, and introduced the notion of a multiple non-Hausdorff mapping cylinder.
\begin{definition}\cite{Barmak}\label{h}
    Let $X_0, X_1, \dots, X_n$ be a sequence of finite $T_0$-spaces and let $f_0, f_1, \dots, f_{n-1}$ be a sequence of maps such that $f_i : X_i \to X_{i+1}$ or $f_i : X_{i+1} \to X_i$. If $f_i: X_i \to X_{i+1}$, we say that $f_i$ \textit{goes right}, and in the other case, we say that it \textit{goes left}. We define the \textbf{multiple non-Hausdorff mapping cylinder} $B(f_0, f_1, \dots, f_{n-1}; X_0, X_1, \dots, X_n)$ as follows:
    
    \noindent The underlying set is the disjoint union $\bigsqcup_{i=0}^n X_i$. We keep the given ordering in each copy $X_i$ and for $x$ and $y$ in different copies, we set $x < y$ in either of the following cases:
\begin{enumerate}
    \item If $x \in X_{2i}$, $y \in X_{2i+1}$ and $f_{2i}(x) \leq y$ or $x \leq f_{2i}(y)$.
    \item If $x \in X_{2i}$, $y \in X_{2i-1}$ and $f_{2i-1}(x) \leq y$ or $x \leq f_{2i-1}(y)$.
\end{enumerate}
\end{definition}
    \noindent Note that the multiple non-Hausdorff mapping cylinder coincides with the non-Hausdorff mapping cylinder (Definition \ref{b}) when $n=1$ and the unique map goes right.
    
    \noindent In \cite{Ximena}, we have studied the nerve theorem for complexes and for finite posets and many of its generalizations. In our main results, we will see some extended versions of these theorems based on a new cover.
\begin{definition}[\cite{Ximena}]\label{i}
	Let $\mathcal{U} = \{U_i\}_{i \in I}$ be a family of subsets of a set $X$. The \emph{nerve} of $\mathcal{U}$, denoted by $\mathcal{N}(\mathcal{U})$, is the simplicial complex whose simplices are the finite subsets $J \subseteq I$ such that
	\[
	\bigcap_{i \in J} U_i \neq \emptyset.
	\]
\end{definition}
\begin{definition}\cite{Ximena}\label{j}
     Let $K$ be a finite simplicial complex (or more generally, a regular CW-complex) and let $\mathcal{U} =\{L_i \}_{i\in I}$ be a family of subcomplexes which cover $K$ (i.e. $\bigcup_{i\in I} L_i = K$). If every intersection of elements of $\mathcal{U}$ is empty or contractible then $\mathcal{U}$ is said to be a good cover of $K$.
\end{definition}
\begin{theorem}\cite{Ximena}\label{k}
    Let $K$ be a finite simplicial complex (or more generally, a regular CW-complex) and let $\mathcal{U} =\{L_i \}_{i\in I}$ be a family of subcomplexes which cover $K$ (i.e. $\bigcup_{i\in I} L_i = K$). If $\mathcal{U}$ is a good cover of $K$ then $K$ and $\mathcal{N}(\mathcal{U})$ have the same (simple) homotopy type.
\end{theorem}
\begin{definition}[\cite{Ximena}]\label{l}
	Let $\mathcal{U}$ be a finite open cover of a topological space. The poset $\chi(\mathcal{U}) = \chi(\mathcal{N}(\mathcal{U}))$, where $\mathcal{N}(\mathcal{U})$ denotes the nerve of $\mathcal{U}$ and $\chi(\mathcal{N}(\mathcal{U}))$ denotes its face poset, is called the \emph{non-Hausdorff nerve} of $\mathcal{U}$.
\end{definition}
\begin{definition}\cite{Ximena}\label{m}
     Let $X$ be a finite space and let $\mathcal{U} =\{U_i \}_{i\in I}$ be a family of open subspaces which cover $X$ (i.e. $\bigcup_{i\in I} U_i = X$). If every intersection of elements of $\mathcal{U}$ is empty or homotopically trivial then $\mathcal{U}$ is said to be a good cover of $X$.
\end{definition}
\begin{theorem}\cite{Ximena}\label{n}
    Let $X$ be a finite space and let $\mathcal{U} =\{U_i \}_{i\in I}$ be a family of open subspaces which cover $X$ (i.e. $\bigcup_{i\in I} U_i = X$). If $\mathcal{U}$ is a good cover of $X$ then $X$ and $\chi(\mathcal{U})$ have the same (simple) homotopy type.
\end{theorem}
    \noindent Two finite simplicial complexes $K$ and $L$ are homotopy equivalent if and only if their associated face posets $\chi(K)$ and $\chi(L)$ are weak homotopy equivalent. Similarly, two finite topological spaces $X$ and $Y$ are weak homotopy equivalent if and only if their associated order complexes $K(X)$ and $K(Y)$ are homotopy equivalent. Therefore, the nerve theorem for simplicial complexes and the nerve theorem for finite spaces are equivalent in this context.
    
    \noindent The following version of the nerve theorem for finite spaces, where the cover is not required to be good, was given by Ximena.
\begin{theorem}\cite{Ximena}\label{20}
	Let $X$ be a finite topological space and let $\mathcal{U} = \{U_i\}_{i \in I}$ be an open cover of $X$. Let $N_0(\mathcal{U})$ be the subspace of the non-Hausdorff nerve $N(\mathcal{U})$ consisting of all homotopically trivial intersections. If for every $x \in X$, the subspace $\mathcal{I}_x$ of $N_0(\mathcal{U})$ of the intersections which contain $x$, is homotopically trivial, then $X$ has the same simple homotopy type as $N_0(\mathcal{U})$.
\end{theorem}
\noindent The following theorem is the simplicial complex counterpart of Theorem \ref{20}, originally stated by Ximena as an analog of Theorem \ref{20} without proof. We include the full proof here for completeness.
\begin{theorem}\label{25'}\cite{Ximena}
	Let \( K \) be a finite simplicial complex (or regular cell complex), and let \( U = \{ L_i \}_{i \in I} \) be a finite family of subcomplexes such that \( \bigcup_{i \in I} L_i = K \). Let \( \mathcal{N}_0(U) \) be the subcomplex of the nerve \( \mathcal{N}(U) \) consisting of simplices representing chains of contractible intersections. If, for every simplex \( \sigma \in K \), the subcomplex \( S_\sigma \subseteq \mathcal{N}_0(U) \), formed by all chains whose members contain \( \sigma \), is contractible, then \( K \) has the same simple homotopy type as \( \mathcal{N}_0(U) \).
\end{theorem}
\begin{proof}
	Let \( \mathcal{U} = \{L_i\}_{i \in I} \) be a finite cover of \( K \) by subcomplexes. Applying the finite space functor \( \chi \), we obtain an open cover \( \mathcal{U}_\chi = \chi\{L_i\}_{i \in I} \) of the finite space \( \chi(K) \).
	
	\noindent By definition, the nerve \( \mathcal{N}(U) \) is a simplicial complex encoding the pattern of intersections among the \( L_i \), and \( \mathcal{N}_0(U) \) is the subcomplex consisting only of those simplices corresponding to chains of contractible intersections. Then, in the finite space setting, the subspace \( \chi(\mathcal{N}_0(U)) \subseteq \chi(N(\mathcal{U})) \) consists of points representing homotopically trivial intersections.
	
	\noindent Now, fix a simplex \( \sigma \in K \). The subcomplex \( S_\sigma \subset \mathcal{N}_0(U) \), consisting of simplices in the nerve whose members all contain \( \sigma \), is contractible by assumption. Therefore, the associated finite space \( \chi(S_\sigma) \subseteq \chi(\mathcal{N}_0(U)) \) is homotopically trivial.
	
	\noindent Applying Theorem \ref{20}, which provides a criterion for simple homotopy equivalence between a finite space and a suitable subspace of its nerve, we conclude that \( \chi(K) \diagup\hspace{-0.17 cm}\searrow \chi(\mathcal{N}_0(U)) \) via simple homotopy equivalence. Since \( \chi \) preserves simple homotopy type, we obtain that the original complex \( K \) and the subcomplex \( \mathcal{N}_0(U) \) share the same simple homotopy type. This completes the proof.
\end{proof}

 \begin{theorem}\label{52}\cite{Ximena}
	Let $X$ be a finite poset of height less or equal to $n$ and let $U = \{ U_i \}_{i \in I}$ be an open cover of $X$ such that $\tilde{N}_0(U)$, the subspace of the reduced non-Hausdorff nerve $\tilde{N}(U)$ of collapsible intersections, is again of height less or equal to $n$. If for every $x \in X$, the subspace $\mathcal{I}_x$ of $\tilde{N}_0(U)$ of the intersections which contain $x$ is collapsible and its height is less or equal than $n - h_X(x)$, then $X \diagup\hspace{-0.17 cm}\searrow^{n+1} \tilde{N}_0(U)$.
\end{theorem}

\section{Main Results}
\subsection{Multiple Cylinder of Relations}

    Here, we define the multiple cylinder of the relations, which is a generalization of the relation cylinder (Definition \ref{d}).
\begin{definition}\label{1}
    Let $X_0,X_1,\dots,X_n$ be a sequence of finite $T_0$-spaces and $R_0, R_1,\dots ,R_{n-1}$ be a sequence of maps such that $R_i \subseteq X_i \times X_{i+1}$ or
    $R_i \subseteq X_{i+1} \times X_{i}$. If $R_i \subseteq X_i \times X_{i+1}$ we say that $R_i$ goes right, and in the other case, we say that it goes left. We define the \textbf {multiple cylinder of relations} $B(R_0, R_1,\dots , R_{n-1}; X_0,X_1,\dots,X_n)$ as follows:
    
    \noindent The underlying set is the disjoint union $\bigsqcup\limits_{i=0}^{n} X_i$. We keep the given ordering in each copy $X_i$ and for $x$ and $y$ in different copies, we set $x \leq y$ in either of the following cases:
\begin{enumerate}
    \renewcommand{\labelenumi}{\normalfont\arabic{enumi}.} 
    \item  If $x \in X_{2i}$, $y \in X_{2i+1}$, $x\leq y$ if $\exists$ $x^{\prime}\in X_{2i}$ and $y^{\prime}\in X_{2i+1}$ such that $x\leq x^{\prime}$ in $X_{2i}$ and $y^{\prime}\leq y$ in $X_{2i+1}$ and $x^{\prime} R_{2i} y^{\prime}$ or $y^{\prime} R_{2i} x^{\prime}$.
    \item  If $x \in X_{2i}$, $y \in X_{2i-1}$, $x\leq y$ if $\exists$ $x^{\prime}\in X_{2i}$ and $y^{\prime}\in X_{2i-1}$ such that $x\leq x^{\prime}$ in $X_{2i}$ and $y^{\prime}\leq y$ in $X_{2i-1}$ and $y^{\prime} R_{2i-1} x^{\prime}$ or $x^{\prime} R_{2i-1} y^{\prime}$.
\end{enumerate}
\end{definition}
    \noindent Note that the multiple cylinder of relations coincides with the relation cylinder when $n=1$ and the unique map goes right.
    
    \noindent When induced by a sequence of order-preserving maps, the multiple cylinder of relations can be regarded as a third space that collapses to two distinct finite spaces as explained in Remark \ref{2}.
    
\begin{example}
Consider three finite $T_0$-spaces:
\[
X_0 = \{a_1 < a_2\}, \quad X_1 = \{b_1 < b_2 < b_3\}, \quad X_2 = \{c_1 < c_2\}.
\]

Let \( R_0 \subseteq X_0 \times X_1 \) and \( R_1 \subseteq X_2 \times X_1 \) be defined as:
\[
R_0 = \{(a_1, b_1), (a_2, b_2)\}, \quad R_1 = \{(c_1, b_2), (c_2, b_3)\}.
\]
Here, \( R_0 \) goes right and \( R_1 \) goes left.

We construct the multiple cylinder of relations:
\[
B(R_0, R_1; X_0, X_1, X_2) = X_0 \sqcup X_1 \sqcup X_2
\]
where the ordering is defined as follows:
\begin{itemize}
    \item The order within each \( X_i \) is preserved.
    \item For \( x \in X_0 \), \( y \in X_1 \), we set \( x \leq y \) if there exist \( x' \geq x \) in \( X_0 \) and \( y' \leq y \) in \( X_1 \) such that \( x' R_0 y' \).
    \item For \( x \in X_2 \), \( y \in X_1 \), we set \( x \leq y \) if there exist \( x' \geq x \) in \( X_2 \) and \( y' \leq y \) in \( X_1 \) such that \( x' R_1 y' \).
\end{itemize}

The resulting partial order includes the following cross-space relations:
\begin{itemize}
    \item \( a_1 \leq b_1, b_2, b_3 \)
    \item \( a_2 \leq b_2, b_3 \)
    \item \( c_1 \leq b_2, b_3 \)
    \item \( c_2 \leq b_3 \)
\end{itemize}

\vspace{1em}
\noindent The Hasse diagram of the space \( B(R_0, R_1; X_0, X_1, X_2) \) is given below:

\begin{center}
\begin{tikzpicture}[scale=1, every node/.style={circle, draw, minimum size=6mm}, node distance=15mm]

\node (a1) at (0,0) {$a_1$};
\node (a2) at (0,1.5) {$a_2$};

\node (b1) at (3,0) {$b_1$};
\node (b2) at (3,1.5) {$b_2$};
\node (b3) at (3,3) {$b_3$};

\node (c1) at (6,0) {$c_1$};
\node (c2) at (6,1.5) {$c_2$};

\draw[->] (a1) -- (a2);
\draw[->] (b1) -- (b2);
\draw[->] (b2) -- (b3);
\draw[->] (c1) -- (c2);

\draw[->, thick, blue] (a1) -- (b1);
\draw[->, thick, blue] (a1) to[bend left=10] (b2);
\draw[->, thick, blue] (a1) to[bend left=15] (b3);
\draw[->, thick, blue] (a2) -- (b2);
\draw[->, thick, blue] (a2) to[bend left=10] (b3);

\draw[->, thick, red] (c1) -- (b2);
\draw[->, thick, red] (c1) to[bend right=10] (b3);
\draw[->, thick, red] (c2) -- (b3);

\node at (0,-1) {\(X_0\)};
\node at (3,-1) {\(X_1\)};
\node at (6,-1) {\(X_2\)};

\end{tikzpicture}
\end{center}    
\end{example}
  
\begin{remark}\label{2}
    Let \( f_0, f_1, \dots, f_{n-1} \) be a sequence of order-preserving maps, such that $f_i$ goes right (or it goes left). These maps induce a corresponding sequence of relations \( R_0, R_1, \dots, R_{n-1} \) defined by \( x R_i f_i(x) \) for every \( x \in X_i \).
    
    In this context, the multiple cylinder of relations \( B(R_0, R_1, \dots, R_{n-1}; X_0, X_1, \dots, X_n) \) aligns with the multiple non-Hausdorff mapping cylinder \( B(f_0, f_1, \dots, f_{n-1}; X_0, X_1, \dots, X_n) \). \\
    Since a multiple non-Hausdorff mapping cylinder can be interpreted as an intermediate space that collapses to two distinct finite spaces (See Sec. 4.6, \cite{Barmak}), the multiple cylinder of relations \( B(R_0, R_1, \dots, R_{n-1}; X_0, X_1, \dots, X_n) \), when induced by a sequence of maps \( f_0, f_1, \dots, f_{n-1} \) as defined above, can similarly be viewed as an intermediate space that collapses to two distinct finite spaces.
\end{remark}
    \noindent The following propositions demonstrate how the collapse results in Propositions \ref{e} and \ref{f} can be extended naturally to compositions of relations.
\begin{proposition}\label{3}
    (\textit{Application of Proposition \ref{e} to Composite Relations}) Let $R_1\subseteq X\times Y$ and $R_2\subseteq Y\times Z$ be relations between finite $T_0$-spaces. If $\underline{(R_2\circ R_1)^{-1}(U_z)}$ is homotopically trivial for every $z\in Z$, then $B(R_2\circ R_1)\diagup\hspace{-0.17 cm}\searrow X$. In particular, $K(B(R_2\circ R_1))$ and $K(X)$ are (simple) homotopy equivalent. If $\underline{(R_2\circ R_1)^{-1}(U_z)}$ is collapsible for all $z\in Z$, then $B(R_2\circ R_1)\searrow X$ and $K(B(R_2\circ R_1))\searrow K(X)$.
\end{proposition}
	
\begin{proof}
    This follows by applying Proposition \ref{e} to the composite relation. Let $z_1,z_2,\dots ,z_n$ be a linear extension of $Z$; that is, if $z_i \leq z_j$ then $i\leq j$. We iteratively remove $z_i$ from $B(R_2\circ R_1)$ by showing each $z_i$ is a $\gamma$-point of $B(R_2\circ R_1)\setminus \{z_1, z_2, \dots, z_{i-1}\}$:
    \[
    \hat{U}_{z_i}^{B(R_2\circ R_1)\setminus \{z_1,z_2,\dots,z_{i-1}\}} = \hat{U}_{z_i}^{B(R_2\circ R_1)}\setminus Z = \underline{(R_2\circ R_1)^{-1}(U_{z_i})},
    \]
    which is homotopically trivial by hypothesis. This shows that, $B(R_2\circ R_1)\diagup\hspace{-0.17 cm}\searrow X$. 
    
    \noindent For the collapsibility case, if $\underline{{(R_2\circ R_1)}^{-1}(U_z)}$ is collapsible, it implies that $\hat{U}_{z_i}^{B(R_2\circ R_1)\setminus \{z_1,z_2,\dots,z_{i-1}\}}$ is collapsible i.e., the link of $z_i$, $\hat{C}_{z_i}^{B(R_2\circ R_1)\setminus \{z_1,z_2,\dots,z_{i-1}\}}$ is also collapsible. Hence (using Remark 4.3.1 and Lemma 4.2.10 in \cite{Barmak}), $K(B(R_2\circ R_1))\searrow K(X)$.
\end{proof}

\begin{proposition}\label{4}
   
	(\textit{Application of Proposition \ref{f} to Composite Relations}) Let $R_1\subseteq X\times Y$ and $R_2\subseteq Y\times Z$ be relations between finite $T_0$-spaces. If $\overline{(R_2\circ R_1)(F_x)}$ is homotopically trivial for every $x\in X$, then $B(R_2\circ R_1)\diagup\hspace{-0.17 cm}\searrow Z$. If $\overline{(R_2\circ R_1)(F_x)}$ is collapsible for all $x\in X$, then $B(R_2\circ R_1)\searrow Z$ and $K(B(R_2\circ R_1))\searrow K(Z)$.
\end{proposition}

\begin{proof}
    This is dual to Proposition \ref{3}. Here, we apply Proposition \ref{f} to the composition. For each $x\in X$, if $\overline{(R_2\circ R_1)(F_x)}$ are homotopically trivial (resp. collapsible), the same argument as in Proposition \ref{3} shows a sequence of collapses to $Z$.
\end{proof}

\begin{remark}\label{5}
	(\textit{Combined Collapse Behavior}) Under the hypotheses of Propositions \ref{3} and \ref{4}, i.e., if both $\underline{(R_2\circ R_1)^{-1}(U_z)}$ and $\overline{(R_2\circ R_1)(F_x)}$ are homotopically trivial for all $z\in Z$, $x\in X$, then $X\diagup\hspace{-0.17 cm}\searrow Z$. Moreover, if these sets are collapsible, then $K(X)\diagup\hspace{-0.17 cm}\searrow^n K(Z)$, where $n = h(B(R_2\circ R_1))$.

\end{remark}

\begin{proposition}\label{6}
    (\textit{Intermediate Collapse via Composition}) Let $R_1\subseteq X\times Y$ and $R_2\subseteq Y\times Z$ be relations between finite $T_0$-spaces. If $\underline{R_1^{-1}(U_y)}$ and $\underline{(R_2\circ R_1)^{-1}(U_z)}$ are homotopically trivial for all $y \in Y$ and, for all $z \in Z$ respectively, then $B(R_2\circ R_1)\diagup\hspace{-0.17 cm}\searrow B(R_1)$. If both are collapsible, then $B(R_2\circ R_1)\diagup\hspace{-0.17 cm}\searrow^n B(R_1)$ with $n = h(X)$.
\end{proposition}
\begin{proof}
\smallskip
	By Proposition \ref{e}, $B(R_1)\diagup\hspace{-0.17 cm}\searrow X$ under the hypothesis on $R_1$. Likewise, Proposition \ref{3} implies $B(R_2\circ R_1)\diagup\hspace{-0.17 cm}\searrow X$. Combining these gives $B(R_2\circ R_1)\diagup\hspace{-0.17 cm}\searrow B(R_1)$.
	
	\noindent Moreover, if $\underline{{R_1}^{-1}(U_y)}$ and $\underline{{(R_2\circ R_1)}^{-1}(U_z)}$ are collapsible for every $y\in Y$ and, for every $z\in Z$ respectively then by  Proposition \ref{e}, $B(R_1)\searrow X$ and by Proposition \ref{3}, $B(R_2\circ R_1)\searrow X$, which gives $B(R_2\circ R_1)\diagup\hspace{-0.17 cm}\searrow^n B(R_1)$ with $n = h(X)$.  
\end{proof}
	
\begin{proposition}\label{7}
(\textit{Intermediate Dual Collapse via Composition}) Let $R_1\subseteq X\times Y$ and $R_2\subseteq Y\times Z$ be relations between finite $T_0$-spaces. If $\overline{R_2(F_y)}$ and $\overline{(R_2\circ R_1)(F_x)}$ are homotopically trivial for all $y \in Y$, $x \in X$, then $B(R_2\circ R_1)\diagup\hspace{-0.17 cm}\searrow B(R_2)$. If both are collapsible, then $B(R_2\circ R_1)\diagup\hspace{-0.17 cm}\searrow^n B(R_2)$ where $n = h(Z)$.    
\end{proposition}	
\begin{proof}
  From Proposition \ref{f}, we have $B(R_2)\diagup\hspace{-0.17 cm}\searrow Z$, and from Proposition \ref{4}, $B(R_2\circ R_1)\diagup\hspace{-0.17 cm}\searrow Z$. Hence, $B(R_2\circ R_1)\diagup\hspace{-0.17 cm}\searrow B(R_2)$. Similar argument can be applied to the collapsibility case as well.   
\end{proof}

\noindent Extending these ideas to multiple compositions, we summarize below the expected collapse behavior for the multiple cylinder of relations.

\begin{remark}\label{8}
	(\textit{Collapse for Multiple Composition}) Let $X_0, X_1, \dots, X_n$ be finite $T_0$-spaces and $R_i \subseteq X_i \times X_{i+1}$ a sequence of relations. If $\underline{(R_{n-1}\circ \cdots \circ R_0)^{-1}(U_{x_n})}$ is homotopically trivial for every $x_n \in X_n$, then $B(R_{n-1} \circ \cdots \circ R_0)\diagup\hspace{-0.17 cm}\searrow X_0$. If collapsible, then $B(R_{n-1} \circ \cdots \circ R_0)\searrow X_0$ and $K(B(R_{n-1} \circ \cdots \circ R_0)) \searrow K(X_0)$.
\end{remark}

\begin{remark}\label{9}
	(\textit{Dual Collapse for Multiple Composition}) In the setup of Remark \ref{8}, if $\overline{(R_{n-1}\circ \cdots \circ R_0)(F_{x_0})}$ is homotopically trivial for all $x_0 \in X_0$, then $B(R_{n-1} \circ \cdots \circ R_0)\diagup\hspace{-0.17 cm}\searrow X_n$. If collapsible, then $B(R_{n-1} \circ \cdots \circ R_0)\searrow X_n$ and $K(B(R_{n-1} \circ \cdots \circ R_0)) \searrow K(X_n)$.
\end{remark}

\begin{remark}\label{10}
	(\textit{Combined Collapse for Multiple Composition}) Under the assumptions of Remarks \ref{8} and \ref{9}, if both $\underline{(R_{n-1}\circ \cdots \circ R_0)^{-1}(U_{x_n})}$ and $\overline{(R_{n-1}\circ \cdots \circ R_0)(F_{x_0})}$ are homotopically trivial for all $x_0 \in X_0$, $x_n \in X_n$, then $X_0\diagup\hspace{-0.17 cm}\searrow X_n$. If collapsible, then $K(X_0)\diagup\hspace{-0.17 cm}\searrow^n K(X_n)$ where $n = h(B(R_{n-1} \circ \cdots \circ R_0))$.
\end{remark}

\subsection{Strong-good covers and its Associated Nerve Theorems}

    \noindent Here we define a new cover for finite simplicial complexes as well as for finite spaces, which gives a stronger version of the nerve theorem for both.
\begin{definition}\label{11}
    Let $K$ be a finite simplicial complex and let $\mathcal{U} =\{U_i \}_{i\in I}$ be a family of open subcomplexes which cover $K$ (i.e. $\bigcup_{i\in I} U_i = K$). If every intersection of elements of $\mathcal{U}$ is empty or collapsible then $\mathcal{U}$ is said to be a \textbf{strong-good cover} of $K$.
\end{definition}

\begin{proposition}\label{12}
 Let $K$ be a finite simplicial complex. If $\mathcal{U}$ is a strong-good cover of $K$, then it is a good cover.   
\end{proposition}
\begin{proof}
 By definition, a strong-good cover $\mathcal{U} = \{U_i\}_{i \in I}$ of $K$ satisfies that each intersection $W_J = \bigcap_{i \in J} U_i$ is either empty or collapsible. Since, collapsible implies simple homotopy equivalent to a point, and for simplicial complexes it is a homotopy equivalence (see p.52, \cite{Barmak}). Therefore, in the simplicial setting, collapsibility implies contractibility.  Thus, each non-empty intersection is contractible, fulfilling the requirement for $\mathcal{U}$ to be a good cover.    
\end{proof}
\noindent However, if $\mathcal{U}$ is a good cover of $K$, then it is not necessarily a strong-good cover of $K$ because not all contractible complexes are collapsible (Example \ref{12a}).

\begin{example}\label{12a}
    The Dunce hat and the Bing's House with two rooms are two contractible complexes without having any free faces. Therefore they are not collapsible.
\end{example}
    \noindent By Proposition \ref{12}, for a simplicial complex, strong-good covers are good covers therefore the nerve theorem for finite simplicial complexes Theorem \ref{k} can be modified for strong-good covers as well:
\begin{remark}\label{13}
    Let $K$ be a finite simplicial complex (or more generally, a regular CW-complex) and let $\mathcal{U} =\{L_i \}_{i\in I}$ be a family of subcomplexes which cover $K$ (i.e. $\bigcup_{i\in I} L_i = K$).If every intersection of elements of $\mathcal{U}$ is empty or collapsible then $K$ and $\mathcal{N}(\mathcal{U})$ have the same (simple) homotopy type.
 \end{remark}
    \noindent Analogously, we define a new open cover for finite spaces that gives the nerve theorem for strong-good covers of finite spaces.
\begin{definition}\label{14}
    Let $X$ be a finite space and let $\mathcal{U} =\{U_i \}_{i\in I}$ be a family of open subspaces which cover $X$ (i.e. $\bigcup_{i\in I} U_i = X$). If every intersection of elements of $\mathcal{U}$ is empty or collapsible then $\mathcal{U}$ is said to be a \textbf{strong-good cover} of $X$.
\end{definition}

\begin{proposition}\label{15}
Let $X$ be a finite $T_0$-space. If $\mathcal{U}$ is a strong-good cover of $X$, then it is a good cover of $X$.
\end{proposition}
\begin{proof}
 By definition, a strong-good cover of $X$ satisfies that every non-empty intersection \( W_J = \bigcap_{i \in J} U_i \) (for \( J \subseteq I \)) is collapsible.
 \noindent In finite spaces, collapsibility implies homotopically triviality, as collapsible spaces can be reduced to a single point by removing \emph{weak points}, whose removal preserves the weak homotopy type of the space (see Proposition 4.2.4 in \cite{Barmak}). Thus, a finite space that is collapsible is weak homotopy equivalent to a point, i.e., homotopically trivial fulfilling the requirement for $\mathcal{U}$ to be a good cover of $X$.

    \noindent However, if $\mathcal{U}$ is a good cover of $X$, then it is not necessarily a strong-good cover of $X$ because homotopically trivial finite spaces do not guarantee that they are collapsible (Example \ref{15a}).    
\end{proof}
    
\begin{example}\label{15a}
    The finite space in Fig 7.3 of \cite{Barmak} is homotopically trivial but not collapsible as it does not have a weak point.
\end{example}
\begin{remark}\label{16}
    If $\mathcal{U}$ is a good cover of $X$, this does not necessarily imply that $\mathcal{U}$ is a strong-good cover of $X$, even if the Whitehead group of every intersection of elements in $\mathcal{U}$ is trivial. The crucial difference is that weak homotopy equivalence does not imply simple homotopy equivalence in finite spaces, nor does it imply collapsibility. A trivial Whitehead group ensures that a space is simple homotopy equivalent to a point, but it does not guarantee collapsibility. Therefore, the conditions for a strong-good cover, which require that every non-empty intersection of elements in $\mathcal{U}$ is collapsible are not satisfied merely by $\mathcal{U}$ being a good cover and having trivial Whitehead groups for all intersections.
\end{remark}
    \noindent By Proposition \ref{15}, strong-good covers for a finite space are good covers therefore the nerve theorem for finite spaces  Theorem \ref{n} can be extended for strong-good covers as well:

\begin{remark}\label{17}
     Let $X$ be a finite space and let $\mathcal{U} =\{U_i \}_{i\in I}$ be a family of open subspaces which cover $X$ (i.e. $\bigcup_{i\in I} U_i = X$).If every intersection of elements of $\mathcal{U}$ is empty or collapsible then $X$ and $\chi(\mathcal{U})$ have the same (simple) homotopy type, where $\chi(\mathcal{U})= \chi(\mathcal{N}(\mathcal{U}))$ is the face poset of the nerve $\mathcal{N}(\mathcal{U})$.
\end{remark}
    \noindent Since two finite simplicial complexes $K$ and $L$ are simple homotopy equivalent if and only if their associated face posets $\chi(K)$ and $\chi(L)$ are simple homotopy equivalent and two finite spaces $X$ and $Y$ are simple homotopy equivalent if and only if their associated order complexes $K(X)$ and $K(Y)$ are simple homotopy equivalent (Theorem 4.2.11 of \cite{Barmak}), therefore we get that the nerve theorem for strong-good covers of complexes and the nerve theorem for strong-good covers of finite space are equivalent to each other.
\begin{remark}\label{19}
    Let $K$ be a finite simplicial complex (or more generally, a regular CW-complex) and let $\mathcal{U} =\{L_i \}_{i\in I}$ be a family of subcomplexes which cover $K$ (i.e. $\bigcup_{i\in I} L_i = K$). If every intersection of elements of $\mathcal{U}$ is empty or homotopically trivial then $\mathcal{U}$ is the same as a good cover of $K$ as for simplicial complexes (or more generally for a regular CW-complex), two simplicial complexes are homotopy equivalent if and only if they are weak homotopy equivalent.
\end{remark}

    \noindent We now demonstrate an application of strong-good covers in the context of persistent homology, showing that under such covers, persistence modules vanish in positive degrees. The results below further distinguish strong-good covers from classical good covers.
    
\begin{lemma}[Union of Collapsible Complexes]
\label{lem:union-collapsible}
Let \( A \) and \( B \) be finite simplicial complexes such that \( A \), \( B \), and \( A \cap B \) are collapsible. Then \( A \cup B \) is collapsible.
\end{lemma}

\begin{proof}
We collapse \(A\cup B\) to a point in three stages, each invoking Proposition 4.1.3 from \cite{Barmak}

\medskip
\noindent\textbf{Step 1. Collapse \(A\) onto \(A\cap B\).}
Since \(A\) is collapsible, there is a sequence of elementary collapses
\[
  A \;\searrow\; \{v\}
\]
for some vertex \(v\in A\).  Choose \(v\in A\cap B\).  By Proposition~4.1.3 applied with
\[
  K = A\cap B,\quad L = A
\]
inside \(A\), we obtain
\[
  A \;\searrow\; A\cap B.
\]

\medskip
\noindent\textbf{Step 2. Collapse \(A\cup B\) onto \(B\).}

Apply Proposition 4.1.3 from \cite{Barmak} in \(A\cup B\) with
\[
  K = B,\quad L = A.
\]
Since \(A\searrow A\cap B\), we deduce
\[
  A\cup B \;\searrow\; B.
\]

\medskip
\noindent\textbf{Step 3. Collapse \(B\) onto \(A\cap B\).}

By symmetry, using Proposition 4.1.3 from \cite{Barmak} with
\[
  K = A\cap B,\quad L = B,
\]
we get
\[
  B \;\searrow\; A\cap B.
\]

\medskip
\noindent\textbf{Step 4. Collapse \(A\cap B\) to a point.}

Finally, since \(A\cap B\) is collapsible by hypothesis, there is a collapse
\[
  A\cap B \;\searrow\; \{v\}.
\]

\medskip
Concatenating these,
\[
  A\cup B
  \;\searrow\; B
  \;\searrow\; A\cap B
  \;\searrow\; \{v\},
\]
shows \(A\cup B\) is collapsible.
\end{proof}

\begin{lemma}[Finite Union of Collapsible Complexes]
\label{lem:finite-union-collapsible}
Let \( \mathcal{A} = \{A_i\}_{i=1}^n \) be a finite collection of simplicial complexes such that:
\begin{enumerate}
    \item Each \( A_i \) is collapsible,
    \item Every non-empty finite intersection \( A_{i_1} \cap \cdots \cap A_{i_k} \) is collapsible for \(2\leq k\leq n\).
\end{enumerate}
Then the union \( \bigcup_{i=1}^n A_i \) is collapsible.
\end{lemma}

\begin{proof}
We proceed by induction on \( n \).

\textbf{Base case (\( n = 1 \)):}  
The union is \( A_1 \), which is collapsible by assumption.

\textbf{Inductive step:}  
Assume the statement holds for all unions of the \( n-1 \) complexes. Let \( A = \bigcup_{i=1}^{n-1} A_i \). According to the induction hypothesis, \( A \) is collapsible. Now consider \( A_n \). We must show that \( A \cup A_n \) is collapsible.

First, observe the intersection \( A \cap A_n = \bigcup_{i=1}^{n-1} (A_i \cap A_n) \). To apply the induction hypothesis to this union, note the following:
(1) \emph{Collapsibility of individual terms:} Each \( A_i \cap A_n \) is collapsible under the second condition of the lemma (as it is a finite intersection of two complexes).

(2) \emph{Collapsibility of intersections:} For any subcollection \( \{A_{i_1} \cap A_n, \dots, A_{i_k} \cap A_n\} \), their intersection is:
\[
\bigcap_{j=1}^k (A_{i_j} \cap A_n) = \left(\bigcap_{j=1}^k A_{i_j}\right) \cap A_n,
\]
which is collapsible according to the second condition of the lemma (since it is a finite intersection of the original family).

By the induction hypothesis, \( A \cap A_n = \bigcup_{i=1}^{n-1} (A_i \cap A_n) \) is collapsible.

Now, since \( A \) is collapsible (by the induction hypothesis), \( A_n \) is collapsible (by assumption), and \( A \cap A_n \) is collapsible (as shown above), we apply Lemma~\ref{lem:union-collapsible} (the two-complex case) to conclude that \( A \cup A_n = \bigcup_{i=1}^n A_i \) is collapsible. 

By induction, the result holds for all \( n \).

\end{proof}

\begin{example}[Two Collapsible Triangles with Collapsible Overlap]
Let \( A \) be the triangle on vertices \( [v_0, v_1, v_2] \) and \( B \) be the triangle on vertices \( [v_0, v_2, v_3] \). The intersection \( A \cap B = [v_0, v_2] \) is an edge. Each triangle is collapsible, and so is their intersection. First, collapse \( A \) to the edge \( [v_0, v_2] \), then collapse \( B \) to the same edge. Finally, collapse the edge to the vertex \( v_0 \). Hence, \( A \cup B \) is collapsible.
\end{example}

\begin{lemma}[Collapsible \(\Rightarrow\) Contractible \(\Rightarrow\) Acyclic]
\label{lem:collapsible-acyclic}
Every collapsible simplicial complex is contractible, and hence has trivial reduced homology in all dimensions; that is,
\[
\widetilde{H}_j(K) = 0 \quad \text{for all } j \geq 0.
\]
\end{lemma}

\begin{proof}
A collapsible complex admits a sequence of elementary collapses reducing it to a single vertex. These collapses define a simple homotopy equivalence between the complex and a point. Since simple homotopy equivalence implies homotopy equivalence, the complex is homotopy equivalent to a point, i.e., contractible. By standard results in algebraic topology, a contractible space has trivial reduced homology in all dimensions.
\end{proof}

\begin{theorem}[Persistent Homology Truncation via Strong Good Cover]\label{persistent-truncation}
Let \( K \) be a finite simplicial complex, and let \( \mathcal{U} = \{U_i\}_{i \in I} \) be a \emph{strong good cover} of \( K \); that is, each \( U_i \) and every non-empty finite intersection
\[
U_\sigma := \bigcap_{i \in \sigma} U_i
\]
is \emph{collapsible}. Then:

\begin{enumerate}
    \item For all \( k \geq 0 \), the inclusion-induced map yields an isomorphism
    \[
    H_k(\mathcal{N}(\mathcal{U})) \cong H_k(K),
    \]
    where \( \mathcal{N}(\mathcal{U}) \) denotes the nerve of the cover \( \mathcal{U} \).

    \item For every subset \( J \subseteq I \), define \( K_J := \bigcup_{i \in J} U_i \). Then
    \[
    H_k(K_J) = 0 \quad \text{for all } k > 0,
    \]
    i.e., the persistence module \( \{ H_k(K_J) \}_{J \subseteq I} \) vanishes in positive degrees.
\end{enumerate}
\end{theorem}

\begin{proof}
\textit{(1)} We apply Meshulam's Homological Nerve Theorem \emph{(\cite{Meshulam} Theorem 2.1)}. It states that if
\[
\widetilde{H}_j(U_\sigma) = 0 \quad \text{for all } \sigma \in \mathcal{N}^{(k)} \text{ and } 0 \leq j \leq k - \dim(\sigma),
\]
then \( \widetilde{H}_j(K) \cong \widetilde{H}_j(\mathcal{N}(\mathcal{U})) \) for all \( 0 \leq j \leq k \), and \( H_{k+1}(\mathcal{N}) \neq 0 \Rightarrow H_{k+1}(K) \neq 0 \).

Since \( \mathcal{U} \) is a strong good cover, every non-empty finite intersection
\[
U_\sigma := \bigcap_{i \in \sigma} U_i
\]
is a collapsible simplicial complex by assumption. By Lemma~\ref{lem:collapsible-acyclic}, this implies that \( U_\sigma \) is contractible and has trivial reduced homology in all degrees:
\[
\widetilde{H}_j(U_\sigma) = 0 \quad \text{for all } j \geq 0.
\]
In particular, this holds for every \( 0 \leq j \leq k - \dim(\sigma) \), so the hypotheses of Meshulam’s theorem are satisfied for all \( k \geq 0 \). Therefore, we conclude that
\[
H_k(K) \cong H_k(\mathcal{N}(\mathcal{U})) \quad \text{for all } k \geq 0.
\]

\textit{(2)} Fix any \( J \subseteq I \). Since \( \mathcal{U} \) is a strong good cover, each \( U_i \) and every finite intersection \( \bigcap_{i \in \sigma} U_i \), with \( \sigma \subseteq J \), is collapsible. Then by Lemma~\ref{lem:finite-union-collapsible}, the union
\[
K_J := \bigcup_{i \in J} U_i
\]
is collapsible. Therefore, \( K_J \simeq \mathrm{pt} \), and it follows that
\[
H_k(K_J) = 0 \quad \text{for all } k > 0.
\]
This shows that the persistence module \( \{H_k(K_J)\}_{J \subseteq I} \) vanishes in all positive degrees.
\end{proof}

\begin{remark}[Failure of Vanishing Under Good Covers]
If \( \mathcal{U} = \{U_i\}_{i \in I} \) is only a \emph{good cover} of \( K \), i.e., each \( U_i \) and all non-empty finite intersections \( U_\sigma \) are \emph{contractible} (but not necessarily collapsible), then the conclusion of Part (2) may fail.

In this case, the nerve theorem still gives a homotopy equivalence:
\[
H_k(\mathcal{N}(\mathcal{U})) \cong H_k(K).
\]
but the space \( K_J \) may have nontrivial homology in positive degrees, depending on the homology of the nerve.
\end{remark}

\begin{example}[Nontrivial Persistent Homology in a Good Cover]
Let \( K = S^1 \), and cover it by three open arcs \( U_1, U_2, U_3 \), each contractible and such that all pairwise and triple intersections are contractible (or empty). Then \( \mathcal{U} = \{U_1, U_2, U_3\} \) is a good cover of \( K \).

The nerve \( \mathcal{N}(\mathcal{U}) \) is the boundary of a 2-simplex (i.e., a circle), so:
\[
H_1(K_J) \cong H_1(\mathcal{N}_J) \cong \mathbb{Z} \quad \text{for } J = \{1,2,3\}.
\]
Hence, the persistence module \( \{H_k(K_J)\}_{J \subseteq I} \) is \emph{nontrivial} in positive degrees. This shows that the vanishing result in Part (2) of the theorem does \emph{not} hold under the weaker assumption of a good cover.
\end{example}

\subsection{Generalizations of Nerve Theorems without the Good Cover Condition}

   \noindent The following result is a variation of Theorem \ref{20}, where we assume collapsibility instead. Though the condition is stronger, the conclusion gives a collapse up to dimension $n$, offering a clearer view of the homotopical structure.
\begin{theorem}\label{21}
    Let $X$ be a finite topological space and let $\mathcal{U} = \{U_i\}_{i \in I}$ be an open cover of $X$. Let $N_0(\mathcal{U})$ be the subspace of the non-Hausdorff nerve $N(\mathcal{U})$ consisting of all collapsible intersections. If for every $x \in X$, the subspace $\mathcal{I}_x$ of $N_0(\mathcal{U})$ of the intersections which contain $x$, is collapsible, then $X \diagup\hspace{-0.17 cm}\searrow^n N_0(\mathcal{U})$.
\end{theorem}
\begin{proof}
    For every $J \subseteq I$, denote $\mathcal{I}_J = \bigcap_{i \in J} U_i$. So the elements of $N_0(\mathcal{U})$ are the subsets $J \subseteq I$ such that $\mathcal{I}_J$ is collapsible.
    
    Define the relation $R \subseteq X \times N_0(\mathcal{U})^{\text{op}}$ as follows: for every $x \in X$ and $J \subseteq I$ such that $\mathcal{I}_J$ is collapsible, set
    \[
    x \, R \, \mathcal{I}_J \iff x \in \mathcal{I}_J.
    \]
    On the one hand,
    \[
    \underline{R^{-1}(U_{\mathcal{I}_J})} = \underline{\mathcal{I}_J} = \mathcal{I}_J,
    \]
    which is collapsible by construction of $N_0(\mathcal{U})$.
    
    On the other hand,
    \[
    \overline{R(F_x)} = \overline{\{ \mathcal{I}_J : \mathcal{I}_J \ni y \text{ for some } y \geq x \}} = \overline{\{ \mathcal{I}_J : \mathcal{I}_J \ni x \}} = \{ \mathcal{I}_J : x \in  \mathcal{I}_J \} = \mathcal{I}_x,
    \]
    which is collapsible by hypothesis.
    
    Thus, by Propositions \ref{e} and \ref{f} and Corollary \ref{g}, $K(X) \diagup\hspace{-0.17 cm}\searrow^n K(N_0(\mathcal{U})^{\text{op}})$ and therefore $X \diagup\hspace{-0.17 cm}\searrow^n N_0(\mathcal{U})^{\text{op}}$. Since $N_0(\mathcal{U})^{\text{op}} \diagup\hspace{-0.17 cm}\searrow N_0(\mathcal{U})$, we deduce that $X \diagup\hspace{-0.17 cm}\searrow^n N_0(\mathcal{U})$.
\end{proof}
    \noindent If we combine Theorem \ref{20} with propositions \ref{e} and \ref{f}, we get the following result: 
\begin{remark}\label{22}
	Let \( X \) and \( Y \) be finite topological spaces, and let \( R_1 \subseteq X \times Y \) be a relation between them. Suppose \( \mathcal{U} = \{U_i\}_{i \in I} \) is an open cover of \( X \), and let \( N_0(\mathcal{U}) \) denote the subspace of the non-Hausdorff nerve \( N(\mathcal{U}) \) consisting of all homotopically trivial intersections.
	
	Assume the following conditions hold:
	\begin{enumerate}
		\item For each \( x \in X \), the subspace \( \mathcal{I}_x \subseteq N_0(\mathcal{U}) \) consisting of intersections that contain \( x \) is homotopically trivial.
		\item For every \( y \in Y \), the set \( \underline{(R_1)^{-1}(U_y)} \subseteq X \) is homotopically trivial.
		\item For every \( x \in X \), the set \( \overline{(R_1)(F_x)} \subseteq Y \) is also homotopically trivial.
	\end{enumerate}
	
	Under these assumptions, we can relate the homotopical structures of \( X \), \( Y \), and \( N_0(\mathcal{U}) \) using known results:
	\begin{enumerate}[label=\alph*.]
		\item By a version of the nerve theorem (Theorem~\ref{20}), condition (1) implies that \( X \) is simple homotopy equivalent to \( N_0(\mathcal{U}) \).
		\item By Proposition~\ref{e}, condition (2) ensures that the poset \( B(R_1) \) collapses to \( X \), so \( X \) and \( B(R_1) \) have the same simple homotopy type.
		\item Similarly, by Proposition~\ref{f}, condition (3) guarantees that \( B(R_1) \) collapses to \( Y \), hence \( Y \) is also simple homotopy equivalent to \( B(R_1) \).
	\end{enumerate}
	
	Putting these together, we see that \( X \), \( Y \), and \( N_0(\mathcal{U}) \) are all simple homotopy equivalent to each other. This remark illustrates how the interaction between a relation and a cover, under suitable homotopical conditions, allows us to transfer the simple homotopy type among the involved spaces.
\end{remark}
    \noindent The previous result can be naturally extended to relate more than two finite spaces via compositions of relations.
    \begin{remark}\label{23}
    	Let \( X \), \( Y \), and \( Z \) be finite topological spaces. Suppose we are given relations \( R_1 \subseteq X \times Y \) and \( R_2 \subseteq Y \times Z \), and define their composition \( R = R_2 \circ R_1 \subseteq X \times Z \). Let \( \mathcal{U} = \{U_i\}_{i \in I} \) be an open cover of \( X \), and let \( N_0(\mathcal{U}) \) denote the subspace of the non-Hausdorff nerve \( N(\mathcal{U}) \) consisting of intersections that are homotopically trivial.
    	
    	Assume the following conditions hold:
    	\begin{enumerate}
    		\item For every \( x \in X \), the subspace \( \mathcal{I}_x \subseteq N_0(\mathcal{U}) \), made up of all intersections that contain \( x \), is homotopically trivial.
    		\item For each \( z \in Z \), the set \( \underline{(R)^{-1}(U_z)} \subseteq X \) is homotopically trivial.
    		\item For each \( x \in X \), the set \( \overline{(R)(F_x)} \subseteq Z \) is homotopically trivial.
    	\end{enumerate}
    	
    	These assumptions allow us to connect the homotopy types of \( X \), \( Z \), and the nerve \( N_0(\mathcal{U}) \), as follows:
    	\begin{enumerate}[label=\alph*.]
    		\item From the first condition, using a nerve-type result (Theorem~\ref{20}), we deduce that \( X \) is simple homotopy equivalent to \( N_0(\mathcal{U}) \).
    		\item The second condition, via Proposition~\ref{3}, implies that the poset \( B(R) \) collapses to \( X \), so \( X \) and \( B(R) \) have the same simple homotopy type.
    		\item Similarly, from the third condition and Proposition~\ref{4}, we get that \( B(R) \) collapses to \( Z \), implying that \( Z \) is also simple homotopy equivalent to \( B(R) \).
    	\end{enumerate}
    	
    	Therefore, putting all this together, we conclude that \( Z \) is simple homotopy equivalent to both \( X \) and \( N_0(\mathcal{U}) \). This observation provides a useful way to understand the homotopical relationships among finite spaces connected via composed relations.
    \end{remark}
    \noindent By combining the observations made in Remarks~\ref{22} and~\ref{23}, we can extend the simple homotopy equivalences across multiple finite topological spaces connected by successive relations.
    \begin{remark}\label{24}
    	Let \( X \), \( Y \), and \( Z \) be finite topological spaces with relations \( R_1 \subseteq X \times Y \) and \( R_2 \subseteq Y \times Z \), and define their composition \( R = R_2 \circ R_1 \subseteq X \times Z \). Suppose \( \mathcal{U} = \{U_i\}_{i \in I} \) is an open cover of \( X \), and let \( N_0(\mathcal{U}) \) be the subspace of the non-Hausdorff nerve \( N(\mathcal{U}) \) consisting of intersections that are homotopically trivial.
    	
    	Assume the following conditions hold:
    	\begin{enumerate}
    		\item For each \( x \in X \), the collection of intersections in \( N_0(\mathcal{U}) \) containing \( x \), denoted \( \mathcal{I}_x \), is homotopically trivial.
    		\item For every \( y \in Y \), the preimage \( \underline{(R_1)^{-1}(U_y)} \subseteq X \) is homotopically trivial.
    		\item For each \( x \in X \), the image \( \overline{(R_1)(F_x)} \subseteq Y \) is homotopically trivial.
    		\item For each \( z \in Z \), the preimage \( \underline{(R)^{-1}(U_z)} \subseteq X \) is homotopically trivial.
    		\item For each \( x \in X \), the image \( \overline{(R)(F_x)} \subseteq Z \) is homotopically trivial.
    	\end{enumerate}
    	
    	With all these conditions satisfied, the following simple homotopy equivalences follow naturally:
    	\begin{enumerate}[label=\alph*.]
    		\item \( Y \) is simple homotopy equivalent to \( N_0(\mathcal{U}) \), and hence to \( X \),
    		\item \( Z \) is also simple homotopy equivalent to \( N_0(\mathcal{U}) \), and therefore to \( X \).
    	\end{enumerate}
    \end{remark}
    \noindent This remark shows how a series of homotopy-based connections and coverings can keep the same simple homotopy type across different finite topological spaces. It points out the clear structure shared by \( X \), \( Y \), and \( Z \) when these conditions are in place.
        
    \noindent The following remark generalizes Remark~\ref{24} to the case of a multiple composition of relations.

\begin{remark}[Generalization to Multiple Cylinders of Relations]\label{25}
	 Consider a finite sequence of spaces \( X_0, X_1, \dots, X_n \), along with relations \( R_i \subseteq X_i \times X_{i+1} \) for \( i = 0, \dots, n-1 \). Let \( \mathcal{U} = \{U_i\}_{i \in I} \) be an open cover of \( X_0 \), and let \( N_0(\mathcal{U}) \) denote the subspace of the non-Hausdorff nerve \( N(\mathcal{U}) \) consisting of homotopically trivial intersections.
	
	Assume the following conditions hold:
	\begin{enumerate}
		\item For every \( x \in X_0 \), the subspace \( \mathcal{I}_x \subseteq N_0(\mathcal{U}) \), consisting of intersections containing \( x \), is homotopically trivial.
		\item For each \( i = 0, \dots, n-1 \) and each \( x_{i+1} \in X_{i+1} \), the subspace \( \underline{(R_i)^{-1}(U_{x_{i+1}})} \subseteq X_i \) is homotopically trivial.
		\item For each \( i = 0, \dots, n-1 \) and each \( x_i \in X_i \), the subspace \( \overline{(R_i)(F_{x_i})} \subseteq X_{i+1} \) is homotopically trivial.
		\item Let \( R = R_{n-1} \circ \cdots \circ R_0 \). Then for each \( x_n \in X_n \), the subspace \( \underline{R^{-1}(U_{x_n})} \subseteq X_0 \) is homotopically trivial, and for each \( x_0 \in X_0 \), the subspace \( \overline{R(F_{x_0})} \subseteq X_n \) is homotopically trivial.
	\end{enumerate}
	
	Under these assumptions, all spaces \( X_0, X_1, \dots, X_n \) are simple homotopy equivalent to the subspace \( N_0(\mathcal{U}) \). In particular, they all share the same simple homotopy type.
\end{remark}
    \noindent This means that if we check certain simple conditions at each step$-$like making sure that the images and preimages of points under these relations behave nicely (in a homotopically trivial way)$-$then the overall shape, or homotopy type, of all these spaces remains the same. So even though we are moving from one space to another through a chain of relations, the essential topological information is preserved. It shows how simple homotopy type can be tracked through such a sequence, starting from an open cover of the first space.

\subsection{Implications of Reduced Non-Hausdorff Nerve in Simplicial Setting}
    \noindent Here we have introduced reduced nerve for simplicial complexes and found its corresponding nerve theorem.
\begin{definition}
    Let \( K \) be a simplicial complex (or more generally, a regular CW-complex) and let \( U = \{L_i\}_{i \in I} \) be a cover of \( K \) by subcomplexes. Consider the open cover \( U_\chi = \{\chi(L_i)\}_{i \in I} \) of the face poset \( X(K) \). We define the reduced nerve \( \tilde{\mathcal{N}}(U) \) as the regular CW-complex satisfying \( \chi(\tilde{\mathcal{N}}(U)) = \tilde{N}(U) \).
\end{definition}
    \noindent The next theorem is the counterpart of the Theorem \ref{52} in the context of complexes.
\begin{theorem}\label{55}
    Let $K$ be a finite simplicial complex (or a regular cell complex) of dimension $\leq n$, and let $\mathcal{U} = \{ L_i \}_{i \in I}$ be a finite family of subcomplexes of $K$ such that $\tilde{ \mathcal{N}}_0(\mathcal{U})$, the subcomplex of the reduced nerve $\mathcal{\tilde{N}(U)}$ whose simplices are chains of collapsible intersections, is again of dimension less or equal to $n$. If for every simplex $\sigma \in K$, the subcomplex $S_\sigma$ of $\tilde{\mathcal{N}}_0(\mathcal{U})$ of chains with elements all containing $\sigma$ is collapsible and its dimension is less or equal than $n - \dim_K(\sigma)$, then $K \diagup\hspace{-0.17 cm}\searrow^{n+1} \tilde{\mathcal{N}}_0(\mathcal{U})$  
\end{theorem}
\begin{proof}
    Let \( \chi(\tilde{\mathcal{N}}(\mathcal{U}))= \tilde{N}(\mathcal{U}) \) such that $\chi(\tilde{\mathcal{N}}(\mathcal{U}))$ denote the reduced nerve of the family \( \mathcal{U} = \{ L_i \}_{i \in I} \), with simplices representing non-empty intersections among the subcomplexes \( L_i \) of $K$. Define \( \tilde{\mathcal{N}}_0(\mathcal{U}) \) as the subcomplex of \( \tilde{\mathcal{N}}(\mathcal{U}) \) whose simplices are chains of collapsible intersections, which has dimension \( \leq n \).
    
    Therefore, in the associated finite space $\chi(K)$, \( \chi(\tilde{\mathcal{N}}_0(\mathcal{U})) \) is the subspace of \( \tilde{N}(\mathcal{U}) \) consisting of all collapsible intersections.

    Now, for each simplex \( \sigma \in K \), let \( S_\sigma \subset \tilde{\mathcal{N}}_0(\mathcal{U}) \) denote the subcomplex containing chains where each element includes \( \sigma \). By hypothesis, \( S_\sigma \) is collapsible and has dimension \( \leq n - \dim_K(\sigma) \). Therefore we have for every $\chi(\sigma)\in \chi(K)$, the subspace $\chi(S_\sigma)$ of $\chi(\tilde{\mathcal{N}}_0(\mathcal{U}))$ of the intersections which contains $\chi(\sigma)$ is collapsible and of height  $ \leq n- h_\chi(K) \chi(\sigma)$.

    By applying Theorem \ref{52}, we can conclude that \( \chi(K) \diagup\hspace{-0.17 cm}\searrow^{n+1} \chi(\tilde{\mathcal{N}}_0(\mathcal{U}))\). This implies that \( K \diagup\hspace{-0.17 cm}\searrow^{n+1} \tilde{\mathcal{N}}_0(\mathcal{U})\). This completes the proof.
\end{proof}

\section{Conclusion}
    In this paper, we introduce multiple cylinder of relations, extending relation cylinders and non-Hausdorff mapping cylinders to sequences of finite $T_{0}$-spaces. When induced by maps, these cylinders serve as intermediate spaces collapsing to the related finite spaces.
    
    We also define strong-good covers for simplicial complexes and finite spaces, with collapsible intersections, and establish a strengthened Nerve Theorem. This shows that such spaces have the same simple homotopy type as their nerves, enhancing homotopical tools. Future work could explore broader uses of multiple cylinder relations and strong-good covers in analyzing complex homotopy structures across finite and combinatorial spaces.

\end{document}